\documentclass[10pt]{amsart}

\usepackage{amsfonts,amsmath,latexsym,amssymb,verbatim,amsbsy,amsthm}
\usepackage{dsfont}

\usepackage{graphicx}

\usepackage{caption}
\usepackage{float}
\usepackage{wrapfig}
\usepackage{blkarray}

\usepackage[top=0.5in, bottom=0.5in, left=0.5in, right=0.5in]{geometry}

\usepackage[dvipsnames]{xcolor}

\usepackage[colorlinks=true, pdfstartview=FitV, linkcolor=RoyalBlue,citecolor=ForestGreen, urlcolor=blue]{hyperref}

\theoremstyle{plain}
\newtheorem{THEOREM}{Theorem}[section]

\newtheorem{theorem}[THEOREM]{Theorem}
\newtheorem{corollary}[THEOREM]{Corollary}
\newtheorem{lemma}[THEOREM]{Lemma}

\newtheorem{assumption}[THEOREM]{Assumption}

\theoremstyle{definition}

\theoremstyle{remark}

\newtheorem{remark}[THEOREM]{Remark}

\def \a {\alpha}

\def \d {\delta}
\def \k {\kappa}
\def \e {\varepsilon}

\def \l {\lambda}

\def \th {\theta}
\def \o {\omega}
\def \w {\omega}

\def \D {\Delta}

\def \L {\Lambda}

\def \O {\Omega}

\def \bz {{\bf z}}

\def \cL {\mathcal{L}}

\newcommand{\R}{\ensuremath{\mathbb{R}}}   

\def \dx  {\, \mbox{d}x}

\def \ds  {\, \mbox{d}s}

\def \ddt  {\frac{\mbox{d\,\,}}{\mbox{d}t}}

\def \Re {\mathrm{Re}}
\def \Im {\mathrm{Im}}

\title{Lyapunov stability and exponential phase-locking of Schr\"odinger-Lohe quantum oscillators}

\author[P. Antonelli]{Paolo Antonelli}
\address{Gran Sasso Science Institute, viale Francesco Crispi, 7, 67100 L'Aquila, Italy}
\email{paolo.antonelli@gssi.it}

\author[D. Reynolds]{David N. Reynolds}
\address{Departamento de Matemática Aplicada and Modeling Nature (MNat) Research Unit. Facultad de Ciencias. Universidad de Granada. 18071-Granada. Spain.}
\email{reyndn12@go.ugr.es}

\date{\today}

\begin{document}

\maketitle
\begin{abstract}
We study the well known Schr\"odinger-Lohe model for quantum synchronization with non-identical natural frequencies. The main results are related to the characterization and convergence to phase-locked states for this quantum system. The results of this article are four-fold. Via a characterization of the fixed points of the system of correlations, we uncover a direct correspondence to the fixed points of the classical Kuramoto model. Depending on the coupling strength, $\k$, relative to natural frequencies, $\O_j$, a Lyapunov function is revealed which drives the system to the phase-locked state exponentially fast. Explicit bounds on the asymptotic configurations are granted via a parametric analysis. Finally, linear stability (instability) of the fixed points is provided via an eigenvalue perturbation argument.

Although the Lyapunov and linear stability are related, their arguments and results are of a different nature. The Lyapunov stability provides a specific value $\k(\O_j)$, where for $\k>\k(\O_j)$, there exists a set of initial data, quantitatively defined, such that the system relaxes to the fixed point at a quantitatively defined exponential rate. While the linear stability is given by analysis of the Jacobian of the fixed points for values of $\k$ large, considering $\varepsilon=\frac{1}{\k}$ as the perturbation parameter. Under certain assumptions this provides stability for a wider range of $\k$ than that which is given in the Lyapunov argument.\\
 
\end{abstract}

\section{Introduction}
 Spontaneous synchronization occurs when any number of individuals act in unison. There are countless examples of this phenomenon. Crickets chirping, fireflies flashing their lights, humans clapping, pacemaker cells in the heart, coupled pendulums, and even the way the human brain functions, can all be modeled via synchronous systems. Given the plethora of examples, it is not surprising that numerous researchers from diverse backgrounds have worked to provide a greater understanding of such systems. From the mathematical perspective, the paradigm model for describing the synchronization of many body systems is the Kuramoto model \cite{KUR}. The standard system tracks the phases, $\th_j$, of $N$ oscillators evolving on the unit sphere, $\mathbb{T}$, according to a competition between each oscillators own natural frequency, $\omega_j$, and a sinusoidal coupling between all other oscillators, modulated by a coupling parameter $K>0$. The model is given by
 \begin{align}\label{e:kur1}
    \ddt \th_j&=\o_j+\frac{K}{N}\sum_{l=1}^N\sin(\th_l-\th_j), \ \ \ j=1,...,N.
\end{align}
The Kuramoto model has received much attention since its inception in 1975. Providing a comprehensive list of references would be prohibitive, but the following articles give the relevant background for this article \cite{ABPRS,DF,MS,Str}. Of particular importance, are the many amenable qualities the Kuramoto model exhibits, which allows for a rigorous mathematical analysis. Indeed, the system is rotationally invariant, and satisfies a gradient flow, which in some regimes can be proved to be convex, granting exponential convergence to the phase-locked state. Further, its simplicity allows for easy implementation for numerical simulations. For many of these reasons, the Kuramoto model has served as a basis for incorporating other features that are often seen in synchronous systems, for instance, Hebbian dynamics \cite{PPS}, frustration parameters \cite{BK}, higher order interactions \cite{MTB}, recent applications to machine learning \cite{NHSNNN}, and adaptation to the quantum setting \cite{DV}.

The natural question that Lohe \cite{LH2} addressed was whether we could observe similar synchronization behavior in higher-dimensional generalizations of the Kuramoto model, where the oscillators can move on a $d-$dimensional sphere. The model therein is the so-called Lohe matrix model,
\begin{align}\label{e:LoheM}
    i(\ddt U_j)U_j^*=H_j+\frac{ik}{2N}\sum_{l=1}^NU_jU_l^*-U_lU_j^*,
\end{align}
with $U_j$ being $d\times d$ unitary matrices, and $H_j$ a Hermitian $d\times d$ matrix whose eigenvalues are the natural frequencies of the j-th oscillator. The model can also be interpreted as a system of Schr\"odinger equations for wave functions attaining only $d$ quantum energy levels. A recent work revealed a fibered gradient flow structure that allows for a similar relaxation to phase-locked states for the Lohe matrix model \eqref{e:LoheM}, as for the Kuramoto model \eqref{e:kur1}, \cite{PP}. Further, an analogous mean field theory was provided in \cite{GH}, putting the standard theory of the Lohe matrix model on a similar level as that of the Kuramoto model, highlighting the suitability of being used as a higher-dimensional generalization. The model which we will largely be concerned with in this article was put forward in \cite{LH1} where the quantum mechanical formulation was formulated in the infinite dimensional case, namely by considering the wave functions to belong to the Hilbert space $L^2(\R^d)$.

The Schr\"odinger-Lohe model may be written as
\begin{align}\label{eq:SL_or}
i\partial_t \psi_j&=H_j\psi_j+\frac{i\k}{2N}\sum_{l=1}^N\left(\psi_l-\frac{\langle \psi_l,\psi_j\rangle}{\|\psi_j\|_2^2} \psi_j\right),\\
\psi_j(x,0)&=\psi_j^0(x), \ \ j=1,\dots,N, \nonumber
\end{align}
with coupling strength $\k>0$, each $\psi_j^0 \in L^2(\R^d)$, and $\langle f,g \rangle=\int_{\R^d} \overline{f(x)}g(x) \dx$ is the standard inner product on $L^2(\R^d)$. Further, each $H_j$ are Hamiltonian operators, that we assume to have the form $H_j=H+\O_j$, where the real constants $\O_j$ play the same role of the natural frequencies in the classical Kuramoto model.

The Schr\"odinger-Lohe model describes a system of $N$ quantum oscillators, all connected to each other, nonlinearly interacting among themselves through unitary transformations. In particular this means that the $L^2$ norm of each wave function is conserved in time (as it can also be checked by direct calculation), despite the presence of the non-Hamiltonian terms in \eqref{eq:SL_or} that drive the system towards synchronization.

Although system \eqref{eq:SL_or} was not derived as a relevant model for specific applications, it is a relatively simple generalization of the Kuramoto system to a quantum context, that serves as a good theoretical starting point to investigate synchronization in quantum systems. 
In fact, the modeling of some quantum circuits suggests that nonlinearities in connecting optical fibers could be described by the nonlinear effects encoded in \eqref{eq:SL_or}, see \cite{LNNTB, Nigg} for related discussions.
Moreover, the study of quantum synchronization models as in \eqref{eq:SL_or} could provide some insights in the study of quantum coherence and entanglement, with applications towards quantum computations, communication and metrology \cite{Kim}.

The Schr\"odinger-Lohe model \eqref{eq:SL_or} and the asymptotic behavior of its solutions was rigorously studied in \cite{AM, HHK}, see also \cite{CCH, HK_frust}. In particular, the case of two oscillators is well understood, as it is possible to show complete frequency synchronization under general assumptions on the initial data. On the other hand, the general system \eqref{eq:SL_or} with $N\geq3$ presents major difficulties so that complete results are available only for identical oscillators, i.e. $\Omega_j=0$, for all $j=1, ..., N$, see also \cite{HHK}. In this case, under some assumptions on the initial data, the system experiences complete phase synchronization. Within \cite{AR} via a coupling of frustration parameters controlled by the Cucker-Smale protocol, the authors provided a result for synchronization of Schr\"odinger-Lohe type oscillators which are nonidentical. However, the result therein utilized the well-known Cucker-Smale theory and does not apply to the original model \eqref{eq:SL_or} with heterogeneous natural frequencies, $\O_j\neq \O_k$, for some $j\neq k$.

In order to bring the basic theory of the Schr\"odinger-Lohe model to the same level of the Kuramoto model and the Lohe matrix model there are two major open problems:
\begin{enumerate}
    \item[OP1] Relaxation to phase-locked states for $\k>\k^*$.
    \item[OP2] Establishment of a rigorous mean-field limit as $N\to \infty$.
\end{enumerate}

It should be noted that the authors in \cite{HHK_disc} gave a partial answer to OP1, yielding convergence to a phase-locked state utilizing the fact that the system of correlations \eqref{correlations} is autonomous. However, they do not characterize the asymptotic state and our argument significantly improves the permissible values of $\k$ for which phase-locking occurs. We provide characterization of the fixed point in Lemma \ref{l:phase-locked} and Theorem \ref{Kur-SL} provides a direct correspondence between the fixed points of the Kuramoto model and the fixed points of the correlation system related to the Schr\"odinger-Lohe model. We then provide a more complete answer to OP1, via Theorem \ref{Lyapunov}, Theorem \ref{t:asymptotic}, and Theorem \ref{t:unique}. In particular, the condition for phase-locking provided in \cite{HHK_disc} is that $\k>4\max_{j,k}|\O_j-\O_k|$, while in Theorem \ref{Lyapunov} we push this further to $\k> (2\sqrt{\frac{11+8\sqrt{2}}{7}})\max_j |\O_j| \sim 3.57 \max_j |\O_j|$. Where, $\max_j |\O_j|<\max_{j,k}|\O_j-\O_k|$ since it is always assumed that $\sum_{j=1}^N\O_j=0$. This is further improved within Theorem \ref{t:unique} to the condition $\k>2\sqrt{2}\max_j |\O_j| \sim 2.83 \max_j |\O_j|$.

The OP2 will need a new mean-field theory that as far as the authors are aware does not yet exist. In particular, the result of Golse and Ha in \cite {GH} which gives the mean-field theory for the Lohe matrix model cannot be readily transferred to the Schr\"odinger-Lohe model due to the infinite dimensional nature of system \eqref{eq:SL_or}. This problem will remain open and will be the topic of a future study.

The remainder of the paper is structured as follows. Section \ref{S:prelim} will give more background on the Schr\"odinger-Lohe model and provide the necessary notation. Section \ref{NNL} provides a characterization of the fixed points of the system of correlations \eqref{correlations}, culminating in Theorem \ref{Kur-SL}. Section \ref{S:Lyapunov} introduces a Lyapunov function which in Theorem \ref{Lyapunov} is shown to satisfy a Gr\"onwall inequality, providing exponential relaxation to the phase-locked state characterized in Lemma \ref{l:phase-locked}. Section \ref{S:asym} provides explicit convergence in $L^2$ and yields the existence of asymptotic wave functions, $\tilde{\psi}_j$, characterized in Theorem \ref{t:asymptotic}. Section \ref{AV} gives estimates on the asymptotic values of the phase-locked state. In particular, bounds on the order parameter, and the angle of opening for the system are provided in Lemmas \ref{l:opening} and \ref{k*bound}. Section \ref{S:Stable} uses the bounds obtained in Section \ref{S:Lyapunov} and Section \ref{AV} in order to prove linear stability (instability) of the fixed points of \eqref{correlations}, which can be characterized via Lemma \ref{l:phase-locked}, via an eigenvalue perturbation argument.

\section{Preliminaries}\label{S:prelim}
Consider the  Schr\"odinger-Lohe model \eqref{eq:SL_or}. It is known, and easy to check directly, that the $L^2$ norm of each wave function is conserved in time. Therefore we let $\|\psi_j\|_2\equiv 1$ and rewrite the equation without the normalization in the nonlinear part of the coupling.
\begin{align}\label{SLmain}
&i\partial_t\psi_j=H\psi_j+\O_j\psi_j+\frac{i\kappa}{2N}\sum_{l=1}^N\left(\psi_l-\langle\psi_l,\psi_j\rangle\psi_j\right), \ \ \ j=1,...,N\\
   & \frac{1}{N}\sum_{j=1}^N \O_j=0,\label{rotateinvar}
\end{align}
with $H=-\frac{1}{2}\D+V$, for $\D$ the usual Laplacian and $V=V(x)$ could be a large class of potentials for which global well-posedness of the system \eqref{eq:SL_or} is well known. Physically relevant examples include the harmonic oscillator $V(x)=\frac{\omega^2}{2}|x|^2$, or any potential $V(x)\sim|x|^{-\gamma}$ that is longer range than the Coulomb interaction $\gamma<1$, or any linear combination of such potentials, see \cite{AR}. The equation \eqref{rotateinvar} is assumed to hold via moving to the rotational reference frame via the rotational invariance of the system. The system \eqref{SLmain} has the order parameter $\zeta=\frac{1}{N}\sum_{j=1}^N\psi_j$, which gives an equivalent formulation
\begin{align}
&i\partial_t\psi_j=H\psi_j+\O_j\psi_j+\frac{i\kappa}{2}\left(\zeta-\langle\zeta,\psi_j\rangle\psi_j\right), \ \ \ j=1,...,N\\
   & \frac{1}{N}\sum_{j=1}^N \O_j=0,
\end{align}

The order parameter can be used as a measure of how synchronized the system is. Indeed, $\|\zeta\|_2 \sim 1$ implies the system is synchronizing, and $\|\zeta\|\sim 0$ that the system is close to incoherence. This order parameter was used within \cite{HHK} to give a monotonicity argument leading to full synchronization of the system \eqref{eq:SL_or} in the homogenous regime $(\O_j\equiv 0)$.

Analogous to the Kuramoto model, whether or not an asymptotic state exists is dependent upon how large the coupling strength $\k>0$ is, relative to the set of natural frequencies $\{\O_j\}_{j=1}^N$. Indeed, the case of $N=2$ oscillators has been treated in \cite{AM} where it was determined that for $\L=\frac{(\O_1-\O_2)}{k}$, whenever $0\leq  \L\leq 1$ synchronous behavior occurs, while for $\L>1$, the behavior is periodic and no asymptotic state exists.\\

Indeed, the $N=2$ result is achieved by analyzing the correlation function $z=\langle \psi_1,\psi_2\rangle$, and its corresponding time derivative,
\begin{align}
    \ddt z=\frac{\k}{2}(1-2i\Lambda z-z^2),
\end{align}
which is a complex valued Riccati equation granting an explicit solution, which gives convergence to the phase locked state if and only if $0\leq \Lambda \leq 1$, and at an exponential rate if $\Lambda<1$. Attempting to utilize the same technique for a larger number of wave functions has proved difficult as for general $N$ the number of correlation functions is on the order of $N^2$. Indeed, for general $N$, letting $z_{jk}=\langle \psi_j,\psi_k\rangle$, you get the system
\begin{align}\label{correlations}
    \ddt z_{jk}&=i(\O_j-\O_k)z_{jk} +\frac{\k}{2N}\sum_{l=1}^N(z_{jl}+z_{lk})(1-z_{jk}), \ \ j,k=1,...,N.
\end{align}
The original system \eqref{eq:SL_or} is a system of $N$ partial differential equations. In order to reduce to more tenable ordinary differential equations, we must pay the price of controlling a system of $\frac{N(N-1)}{2}$ variables. This is the main reason why the system is well understood for $N=2$, while having little theory for $N\geq 3$.

However, there is structure to the system that can be taken advantage of. It is well known for the Schr\"odinger-Lohe system that the mass is preserved, i.e.
\begin{align}
    \ddt \|\psi_j\|_2=0,
\end{align}
which leads us to take $\|\psi_{j,0}\|_2\equiv 1$ for each $j=1,...N$. Hence, by Cauchy-Schwartz automatically the correlation functions $z_{jk}$ are confined to the unit ball within the complex plane:
\begin{align}
    |z_{jk}|=|\langle \psi_j,\psi_k\rangle|\leq \|\psi_j\|_2\|\psi_k\|_2=1,
\end{align}
therefore $z_{jk}(t)\in B_1(0)\subset \mathbb{C}$, for each $j,k=1,...,N$, and in particular $z_{jj}=\|\psi_j\|_2^2=1$. Therefore, the system of correlations \eqref{correlations} is an autonomous system with dynamics given by holomorphic functions confined to the unit ball within the complex plane.\\

Notation used throughout the remainder of the paper include, $z_{jk}=\langle \psi_j,\psi_k\rangle=\int_{\mathbb{R}^d}\overline{\psi_j(x)}\psi_k(x) \dx.$ The order parameter is given by $ \zeta=\frac{1}{N}\sum_{j=1}^N\psi_j,$ and correlations with the order parameter, $ z_j=\langle \zeta,\psi_j\rangle$. All norms and inner products are within $L^2(\R^d)$.
Further, the real and imaginary parts of the correlation functions will be denoted $\Re z_{jk}=r_{jk}, \Im z_{jk}=s_{jk}$ and similarly for $\Re z_{j}=r_{j}, \Im z_{j}=s_{j}$, and the $L^2$ norm of the order parameter denoted $\lambda(t)=\|\zeta\|_2(t)$. Other notation will be introduced where necessary.
\section{Characterization of the phase-locked state}\label{NNL}
In order to achieve convergence to the phase-locked state, we must first characterize what such a fixed point looks like. We take inspiration from the fact that in equation \eqref{correlations}, the Hamiltonian plays no role. We therefore construct new wave functions, $\psi_{j,\infty}(x)$, $j=1,...,N$, such that $\|\psi_{j,\infty}\|_2=1$ and where these functions solve the following stationary equations, excluding the Hamiltonian:
\begin{align}\label{SLstat}
0=\O_j\psi_{j,\infty}+\frac{i\kappa}{2}\left(\zeta_{\infty}-\langle\zeta_{\infty},\psi_{j,\infty}\rangle\psi_{j,\infty}\right), \ \ \ j=1,...,N.
\end{align}
We will utilize these wave functions, and their correlations in order to extract the asymptotic states for the system \eqref{correlations}. In the remainder of the paper, any index of $\infty$ will denote any functions, correlations, etc, whose information was derived from the system of stationary equations \eqref{SLstat}.

\begin{remark}  
Seeking solutions to \eqref{SLstat} is further justified by Theorem 1.1 of \cite{HH} which decouples the nonlinear PDE system \eqref{eq:SL_or} into a system of linear Schr\"odinger equations with the nonlinear part hidden in a system of ODEs (analagous to system \eqref{correlations}). The result therein is only given for the homogeneous system $\O_j\equiv 0$, but the requisite changes to the inhomogeneous case studied here are straightforward. We do not present this adaptation so as not to introduce new notation, however we make note that this process allows us to find a specific class of fixed points to \eqref{correlations}, but it does not guarantee that these are the only potential fixed points of the system.
\end{remark}

Now, starting with \eqref{SLstat} and multiplying by $\bar{\psi}_{j,\infty}$, and integrating, we achieve
\begin{align}
&0=\int_{\R^d}\bar{\psi}_{j,\infty}\left(\O_j\psi_{j,\infty}+\frac{i\kappa}{2}\left(\zeta_{\infty}-\langle\zeta_{\infty},\psi_{j,\infty}\rangle\psi_{j,\infty}\right)\right) \dx\\
    &0=\O_j+\frac{i\k}{2}(\langle\psi_{j,\infty},\zeta_{\infty}\rangle-\langle\zeta_{\infty},\psi_{j,\infty}\rangle),\\
    &\Im\langle \psi_{j,\infty},\zeta_{\infty}\rangle=\frac{\O_j}{\k}=-s_j^{\infty}.
\end{align}
Automatically, as $|z_j^{\infty}|\leq 1$, and thus $|s_j^{\infty}|\leq 1$ a necessary condition for a phase-locked state consistent with \eqref{SLstat} is that $\k\geq \max_{j=1,...N} |\O_j|$. We write this as a lemma.
\begin{lemma}\label{l:nec}
    Let $z_{jk}$, for $j,k=1,...,N$ be solutions to \eqref{correlations}, then a necessary condition for a phase-locked state, which further satisfies \eqref{SLstat}, to exist is that
    \begin{align}
        \k\geq \max_{j=1,...N} |\O_j|.
    \end{align}
\end{lemma}

Continuing, we multiply by $\bar{\zeta}_{\infty}$ and integrate,
\begin{align}
&0=\int_{\R^d}\bar{\zeta}_{\infty}\left(\O_j\psi_{j,\infty}+\frac{i\kappa}{2}\left(\zeta_{\infty}-\langle\zeta_{\infty},\psi_{j,\infty}\rangle\psi_{j,\infty}\right)\right) \dx\\
    &0=\O_j\langle\zeta_{\infty},\psi_{j,\infty}\rangle+\frac{i\k}{2}(\|\zeta_{\infty}\|_2^2-\langle\zeta_{\infty},\psi_{j,\infty}\rangle^2),
    \end{align}
    The real part does not grant any new information, however the imaginary part does,
    \begin{align}
0&=\O_j\Im\langle\zeta_{\infty},\psi_{j,\infty}\rangle+\frac{\k}{2}(\|\zeta_{\infty}\|_2^2-\Re\langle\zeta_{\infty},\psi_{j,\infty}\rangle^2+\Im\langle\zeta_{\infty},\psi_{j,\infty}\rangle^2),\\
\Re\langle\zeta_{\infty},\psi_{j,\infty}\rangle^2&=\|\zeta_{\infty}\|_2^2-\Im\langle\zeta_{\infty},\psi_{j,\infty}\rangle^2,\\
\l_{\infty}^2&=(r_j^{\infty})^2+(s_j^{\infty})^2
\end{align}
for each $j=1,...,N$. This implies that for each $j=1,...,N$, we have $|\langle\zeta_{\infty},\psi_{j,\infty}\rangle|=\|\zeta_{\infty}\|_2$, and since $\|\psi_{j,\infty}\|_2=1$, by Cauchy-Schwartz, $\zeta_{\infty}=\l_{\infty} e^{i\a_j}\psi_{j,\infty}$, for each $j=1,..,N$, where $\a_j$ denotes how far each individual $\psi_{j,\infty}$ is rotated away from the order parameter $\zeta_{\infty}$, and $\l_{\infty}=\|\zeta_{\infty}\|_2 \in [0,1]$.\\

Let us now rewrite the phase-locked state in terms of the angles $\a_j$ and the norm of the order parameter $\l_{\infty}$.
\begin{lemma}\label{l:phase-locked}
    Let $\psi_{j,\infty},$ for $ j=1,...,N$ solve the stationary equation \eqref{SLstat}. Then for $\a_j \in [-\pi,\pi)$ determined by the parameters $\k,\O_j$, the following relations characterize a class of fixed points of the system \eqref{correlations}:
\begin{align}
    \psi_{j,\infty}&=\frac{\zeta_{\infty}}{\l_{\infty}}e^{-i\a_j},\\
    r_{jk}^{\infty}&=\cos(\a_j-\a_k),\\
    s_{jk}^{\infty}&=\sin(\a_j-\a_k),\\
    z_{jk}^{\infty}&=e^{i(\a_j-\a_k)},\\
    r_j^{\infty}&=\l_{\infty}\cos(\a_j),\\
    s_j^{\infty}&=-\l_{\infty}\sin(\a_j)=-\frac{\O_j}{\k},\label{lOmega}\\
    z_j^{\infty}&=\l_{\infty}e^{-i\a_j},\\
    \l_{\infty}&=\frac{1}{N}\sum_{j=1}^N\cos(\a_j),\label{lcos}\\
    0&=\frac{1}{N}\sum_{j=1}^N\sin(\a_j).
\end{align}
\end{lemma}
Plugging these conditions into \eqref{correlations} yields a fixed point.
Note further that these relations imply that the particular phase-locked states we have found are on the boundary of the unit ball in the complex plane. Indeed, $|z_{jk}^{\infty}|^2=\cos(\a_j-\a_k)^2+\sin(\a_j-\a_k)^2=1$, so that $z_{jk}^{\infty}\in \partial B_1(0)$. Of further interest, when observing the correlation values with the order parameter $z_j^{\infty}$, we can recover the equilibria of the Kuramoto model.

Let $\w_j \in \R$ be fixed natural frequencies and consider the Kuramoto model for $N$ phase-coupled oscillators
\begin{align}
    \ddt \th_j&=\o_j+\frac{K}{N}\sum_{l=1}^N\sin(\th_l-\th_j), \ \ \ j=1,...,N,\\
    \frac{1}{N}\sum_{j=1}^N \o_j&=0.
\end{align}
Let $Re^{i\Phi}=\frac{1}{N}\sum_{j=1}^Ne^{i\th}$ be the order parameter for the Kuramoto model. Then, as we have shifted to the rotational reference frame with $\frac{1}{N}\sum_{j=1}^N \o_j=0$, we know $\Phi=0$, and the system can be reduced to
\begin{align}\label{Kuramoto}
    \ddt \th_j=\o_j-KR\sin(\th_j), \ \ \ j=1,...,N.
\end{align}
Indeed, for large enough $K$, see \cite{MS}, the phase-locked state of this system is given by
\begin{align}\label{Kur:fix}
    \sin(\th_j^{\infty})=\frac{\o_j}{KR_{\infty}}, \ \ \ j=1,...,N,
\end{align}
where $R_{\infty}=\frac{1}{N}\sum_{j=1}^Ne^{i\th_j^{\infty}}=\frac{1}{N}\sum_{j=1}^N\cos(\th_j^{\infty})$, again due to the rotational reference frame.

The following theorem provides a direct correspondence between the fixed points of the Kuramoto model and the fixed points of system \eqref{correlations} which also satsify \eqref{SLstat}. In and of itself, this result is important as the Schr\"odinger-Lohe model was constructed as a generalization of the Kuramoto model to the quantum setting. It is therefore not surprising that we can find this result, however, it will also prove useful in Section \ref{S:Stable} for proving linear stability (instability) results.
\begin{theorem}\label{Kur-SL}
    Let $\o_j=\O_j \in \R$ fixed natural frequencies such that $\sum_j\o_j=0$, and $K=\k>0$ large enough so that phase-locked states exist for both the Kuramoto model \eqref{Kuramoto} and the correlation system derived from the Schr\"odinger-Lohe model \eqref{correlations}, which further satisfies \eqref{SLstat}. Then the asymptotic angular values $\a_j$ can be mapped onto the corresponding phase-locked state $\th_j^{\infty}$ of the Kuramoto model in the following way:
    \begin{align}\label{SLtoKur}
        -\a_j \mapsto \th_j^{\infty},
    \end{align}
    and similarly
    \begin{align}\label{KurtoSL}
        \th_j^{\infty} \mapsto -\a_j.
    \end{align}
    Further, this implies that the phase-locked state described in Lemma \eqref{l:phase-locked} always exists if a phase-locked state of the Kuramoto model exists.
\end{theorem}
\begin{proof}
    As it is assumed that $\o_j=\O_j$ and $K=\k$, we need only see that $\l_{\infty}=R_{\infty}$. We have already shown that
    \begin{align}
        \l_{\infty}=\frac{1}{N}\sum_{j=1}^N\cos(\a_j)=\frac{1}{N}\sum_{j=1}^N\cos(\th_j^{\infty})=R_{\infty},
    \end{align}
then for the variables $y_j^{\infty}=\frac{\bar{z}_j^{\infty}}{\l_{\infty}}=e^{i\a_j}$, $j=1,...,N$, and $\th_j^{\infty}$ given by \eqref{Kur:fix},
    \begin{align}
        \Im(y_j^{\infty})&=\sin(\a_j)=\frac{\O_j}{\k\l_{\infty}},\\
        \sin(\th_j^{\infty})&=\frac{\o_j}{KR_{\infty}}, \ \ \ j=1,...,N.
    \end{align}
    In this way we see that both $e^{i\th_j^{\infty}}$ and $y_j^{\infty}$ are found on the unit sphere with $\Im(y_j^{\infty})=\sin(\th_j^{\infty})$, and hence $y_j^{\infty}=e^{i\th_j^{\infty}}$, and any phase-locked state of the Scrh\"odinger-Lohe correlation model \eqref{correlations}, which also satisfies \eqref{SLstat}, corresponds to the same phase-locked state of the Kuramoto system \eqref{Kuramoto}.
\end{proof}

\begin{remark}
    Theorem \ref{Kur-SL} does not characterize all potential fixed points of the original Schr\"odinger-Lohe system \eqref{eq:SL_or}. However, what it grants is that if $K>K_c$ within the Kuramoto model \eqref{e:kur1}, so that a fixed point exists, then the equivalent Schr\"odinger-Lohe system given by the same parameters will also have a fixed point of the system \eqref{correlations} that further satisfies \eqref{SLstat}, where the same angle relations between the classical oscillators and the quantum oscillators can be found via relations \eqref{SLtoKur} and \eqref{KurtoSL}.
\end{remark}

Despite the direct connection of phase-locked states, proof of convergence for the Schr\"odinger-Lohe system is much more complicated. For Kuramoto, there is no spatial component, all oscillators remain constrained to the unit sphere, and in many regimes is driven by a convex gradient flow. Even though all fixed points of \eqref{correlations} satisfying \eqref{SLstat} lie on the unit sphere, in general the dynamics are found in the interior of the unit ball (excluding assumptions which directly reduce the system to Kuramoto). However, given Theorem \ref{Kur-SL}, we can extract more information about the fixed points derived in Lemma \ref{l:phase-locked}. Indeed, given the rotational invariance of the system we can focus on the case of all $\a_j \in (-\frac{\pi}{2},\frac{\pi}{2})$, so that $\cos(\a_j)>0$ for all $j=1,...,N$. These are the only potentially stable fixed points from those characterized by Lemma \ref{l:phase-locked}, as will be shown later in Section \ref{S:Stable}. In the next section we provide the first quantitative result on convergence to the phase-locked state for the Schr\"odinger-Lohe model with more than two nonidentical natural frequencies.

\section{Lyapunov stability of the phase-locked state}\label{S:Lyapunov}
With the characterization of the fixed points given in the previous section, we proceed by uncovering a Lyapunov function which provides stability and exponential convergence to the phase-locked state.

Before stating the theorem, let us derive several useful equations. Taking the real $(r_{jk})$ and imaginary $(s_{jk})$ parts of \eqref{correlations}, and then averaging each of these $(r_j,s_j)$ gives the following four equations:
\begin{align}
    \ddt r_{jk}&=-(\O_j-\O_k)s_{jk}+\frac{\k}{2}\left((r_j+r_k)(1-r_{jk})+(s_k-s_j)s_{jk}\right),\label{eq:r1}\\
    \ddt s_{jk}&=(\O_j-\O_k)r_{jk}+\frac{\k}{2}\left(-(r_j+r_k)s_{jk}+(s_k-s_j)(1-r_{jk})\right),\label{eq:s1}\\
    \ddt r_j&=\O_js_j-\frac{1}{N}\sum_{l=1}^N\O_ls_{lj}+\frac{\k}{2}\left(r_j-r_j^2+s_j^2+\frac{1}{N}\sum_{l=1}^N(r_l-r_lr_{lj}-s_ls_{lj})\right),\label{eq:r2}\\
    \ddt s_j&=-\O_jr_j+\frac{1}{N}\sum_{l=1}^N\O_lr_{lj}+\frac{\k}{2}\left(s_j-2r_js_j+\frac{1}{N}\sum_{l=1}^N(s_lr_{lj}-r_ls_{lj})\right).\label{eq:s2}
\end{align}
Plugging the asymptotic values previously determined in Lemma \ref{l:phase-locked} into the above system yields a stationary state, as to be expected. To prove convergence of the system to this steady state in some basin of attraction to be determined, we further define
\begin{align}
    f_{jk}&=r_{jk}^{\infty}-r_{jk}, \ \ \ \ g_{jk}=s_{jk}^{\infty}-s_{jk}, \label{aux1}\\ 
    f_j&=r_j^{\infty}-r_j, \ \ \ \ g_j=s_j^{\infty}-s_j.\label{aux2}
\end{align}
The functions, $f_{jk},g_{jk}$ measure how far the real and imaginary parts of the correlations are away from their prescribed asymptotic value, found in Lemma \ref{l:phase-locked}, while $f_j,g_j$ provide the same except for the real and imaginary parts of the macroscopic correlation with the order parameter. The Lyapunov candidate is then given by
\begin{align}\label{def:Lyap}
    \cL=\frac{1}{2N^2}\sum_{j,k=1}^N(f_{jk}^2+g_{jk}^2).
\end{align}
Clearly $\cL(t)\geq 0$. The remainder of the section will be devoted to providing a regime in which we can prove exponential decay of this function via a Gr\"onwall estimate. Of particular importance will be how large the coupling strength $\k$ is with respect to the natural frequencies $\O_j$, and further how close initial data is to the fixed point prescribed in Lemma \ref{l:phase-locked}.

\begin{theorem}\label{Lyapunov}
    Let $z_{jk}$ for $j,k=1,...,N$ be solutions to \eqref{correlations} with $\O_j\in\R$ such that $\sum_j \O_j=0$. Then, if $\k> (2\sqrt{\frac{11+8\sqrt{2}}{7}})\max_j |\O_j|$ and initial data satisfies
    \begin{align}
        &M_1=\sup_{j,k}|f_{jk}|, \ \ \ M_2=\sup_{j,k}|g_{jk}|,\\
        &4N(M_1+M_2)-2M_1<C_1(\k),
    \end{align}
    for a positive constant $C_1(\k)$ to be determined later. The system of correlations \eqref{correlations} undergoes exponential convergence to the phase-locked state determined by Lemma \ref{l:phase-locked}, under the assumption that all $\a_j \in (-\frac{\pi}{2},\frac{\pi}{2})$. The rate of convergence is given by the Gr\"onwall inequality,
    \begin{align}
        \ddt \cL \leq -C_2(\k,N,M_1,M_2)\cL,
    \end{align}
    for $C_2(\k,N,M_1,M_2)$ to be determined below.
\end{theorem}

Before we prove this theorem let us prove a couple useful lemmas. The first provides a lower bound on the asymptotic value of the order parameter, $\|\zeta_{\infty}\|_2=\l_{\infty}$, in relation to the size of $\k$ relative to the natural frequencies $\O_j$.
\begin{lemma}
    If $\k\geq A\max_j |\O_j|$, for some $A\geq 2$, then $\l_{\infty}^2\geq \frac{1+\sqrt{1-\frac{4}{A^2}}}{2}$.
\end{lemma}
\begin{proof}
    As $\l_{\infty}=\frac{1}{N}\sum_{j=1}^N\cos(\a_j)$ we have,
    \begin{align}
        \l_{\infty}\geq \cos\left(\max_j|\a_j|\right).
    \end{align}
    However since $\sin(\a_j)=\frac{-\O_j}{\l_{\infty}\k}$, we have,
    \begin{align}
        |\a_j|=\left|\sin^{-1}\left(\frac{-\O_j}{\l_{\infty}\k}\right)\right|,
    \end{align}
    and therefore
    \begin{align}
        \l_{\infty}&\geq \sqrt{1-\left(\frac{\max_j |\O_j|}{\l_{\infty}\k}\right)^2},\\
        &\geq \sqrt{1-\left(\frac{1}{\l_{\infty}A}\right)^2}
    \end{align}
    Squaring both sides and multiplying through by $\l_{\infty}^2$ gives the following inequality,
    \begin{align}
        \l_{\infty}^4-\l_{\infty}^2+\frac{1}{A^2}\geq 0.
    \end{align}
    This inequality is satisfied when $\l_{\infty}^2$ is greater than the second root of the polynomial, i.e. exactly when
    \begin{align}
        \l_{\infty}^2\geq \frac{1+\sqrt{1-\frac{4}{A^2}}}{2}.
    \end{align}
\end{proof}

The second lemma gives a lower bound on $r_{jk}^{\infty}$, under the same assumptions on $\k$ and $\O_j$ from the previous lemma 
\begin{lemma}\label{l:rbound1}
    If $\k\geq A\max_j |\O_j|$, for some $A\geq 2$, then $r_{jk}^{\infty}\geq 1-\frac{4}{A^2\left(1+\sqrt{1-\frac{4}{A^2}}\right)}$ for all $j,k=1,...,N$ and consequently $r_j^{\infty}\geq 1-\frac{4}{A^2\left(1+\sqrt{1-\frac{4}{A^2}}\right)}$ for all $j=1,...,N$.
\end{lemma}
\begin{proof}
    As $\min_{j,k} r_{j,k}^{\infty}=\min_{j,k}\cos(\a_j-\a_k)\geq \cos(2\max_j|\a_j|)$, we get,
    \begin{align}
        \min_{j,k} r_{j,k}^{\infty}&\geq \cos\left(2\max_j\left|\sin^{-1}\left(\frac{-\O_j}{\l_{\infty}\k}\right)\right| \right)\\
        &=1-2\left(\max_{j}\frac{|\O_j|}{\l_{\infty}\k}\right)^2.
    \end{align}
    Now as $\frac{1+\sqrt{1-\frac{4}{A^2}}}{2}\leq\l^2_{\infty}\leq 1$ and $\k\geq A\max_j |\O_j|$ for $A\geq 2$, we have
    \begin{align}
        \min_{j,k} r_{j,k}^{\infty}&\geq 1-\frac{4}{A^2\left(1+\sqrt{1-\frac{4}{A^2}}\right)}.
    \end{align}
\end{proof}
Of particular importance for the proof of Theorem \ref{t:asymptotic} is when $r_{jk}^{\infty}> 2(\sqrt{2}-1)$, which is given by the following corollary.
\begin{corollary}\label{l:rbound2}
    If $\k> (2\sqrt{\frac{11+8\sqrt{2}}{7}})\max_j |\O_j|$, then $r_{jk}^{\infty}\geq 2(\sqrt{2}-1)$ for all $j,k=1,...,N$ and consequently $r_j^{\infty}> 2(\sqrt{2}-1)$ for all $j=1,...,N$.
\end{corollary}
\begin{proof}
    Let $A=2\sqrt{\frac{11+8\sqrt{2}}{7}}$ in Lemma \ref{l:rbound1}.
\end{proof}
A second value which is key for the proof of Theorem \ref{t:unique} is when $r_{jk}^{\infty}\geq \frac{\sqrt{2}}{2}$, which is given by the following coollary.
\begin{corollary}\label{l:rbound3}
    If $\k\geq 2\sqrt{2}\max_j |\O_j|$, then $r_{jk}^{\infty}\geq \frac{\sqrt{2}}{2}$ for all $j,k=1,...,N$ and consequently $r_j^{\infty}\geq \frac{\sqrt{2}}{2}$ for all $j=1,...,N$.
\end{corollary}
\begin{proof}
    Let $A=2\sqrt{2}$ in Lemma \ref{l:rbound1}.
\end{proof}
The next lemma provides a Lyapunov functional candidate which provides exponential decay for certain initial conditions at time $t=0$.
\begin{lemma}\label{l:decay}
    Let $\k\geq A\max_{j}|\O_j|$ for $A>2\sqrt{\frac{11+8\sqrt{2}}{7}}$. Let $\cL$ be defined as in \eqref{def:Lyap} be a Lyapunov candidate.
    Then if $B=\frac{1}{A^2\left(1+\sqrt{1-\frac{4}{A^2}}\right)}$, there exists a positive constant $C_B=2(1-4\sqrt{B}-4B)<2$ such that if $\sup_{j} f_{j}(0)=M<\frac{C_B}{2}$, then for $C_M=2\k(C_B-2M)>0$ the initial decay of the Lyapunov candidate is given by
    \begin{align}
        \ddt \cL|_{t=0}=-C_M\cL(0)
    \end{align}
\end{lemma}
\begin{proof}

Let $\mathcal{L}=\frac{1}{2N^2}\sum_{j,k=1}^N(f_{jk}^2+g_{jk}^2)$ be a Lyapunov candidate. Differentiating gives
\begin{align}
    \ddt \mathcal{L}=\frac{1}{N^2}\sum_{j,k=1}^N&-f_{jk}[\ddt r_{jk}]-g_{jk}[\ddt s_{jk}],\\
    =\frac{1}{N^2}\sum_{j,k=1}^N&-f_{jk}[-(\O_j-\O_k)s_{jk}+\frac{\k}{2}\left((r_j+r_k)(1-r_{jk})+(s_k-s_j)s_{jk}\right)]\\
   &-g_{jk}[(\O_j-\O_k)r_{jk}+\frac{\k}{2}\left(-(r_j+r_k)s_{jk}+(s_k-s_j)(1-r_{jk})\right)]
\end{align}
Now, we plug for $f_{jk},g_{jk},f_{j},g_{j}$ using \eqref{aux1}-\eqref{aux2}.
\begin{align}
    \ddt \mathcal{L}=\frac{1}{N^2}\sum_{j,k=1}^N&-f_{jk}[-(\O_j-\O_k)(s^{\infty}_{jk}-g_{jk})+\frac{\k}{2}\left((r^{\infty}_j-f_j+r^{\infty}_k-f_k)(1-r^{\infty}_{jk}+f_{jk})+(s^{\infty}_k-g_k-s^{\infty}_j+g_j)(s^{\infty}_{jk}-g_{jk})\right)]\\
   &-g_{jk}[(\O_j-\O_k)(r^{\infty}_{jk}-f_{jk})+\frac{\k}{2}\left(-(r^{\infty}_j-f_j+r^{\infty}_k-f_k)(s^{\infty}_{jk}-g_{jk})+(s^{\infty}_k-g_k-s^{\infty}_j+g_j)(1-r^{\infty}_{jk}+f_{jk})\right)].
\end{align}
Now, as $r_{jk}^{\infty}, s_{jk}^{\infty}, r_j^{\infty}, s_j^{\infty}$ correspond to fixed points of equations \eqref{eq:r1}-\eqref{eq:s2}, we can cancel out many of the terms which make up the right hand sides of these equations (these terms correspond exactly to the terms which are linear in $f_{jk}, g_{jk}, f_j, g_j)$. Further cancellations and utilizing the symmetries of the equations yields,
\begin{align}
    \ddt \mathcal{L}=\frac{\k
    }{N^2}\sum_{j,k=1}^N&(f_{jk}^2+g_{jk}^2)(f_j-r_j^{\infty}+f_k-r_k^{\infty})\\
    &+2f_j(f_{jk}(1-r_{jk}^{\infty})-s_{jk}^{\infty}g_{jk})+2g_j(g_{jk}(1-r_{jk}^{\infty})+f_{jk}s_{jk}^{\infty}),
\end{align}
where the first term appears to be giving the proper negative sign for the Gr\"onwall estimate, when the system is close to the fixed point where $f_j,f_k$ will be close to zero. We use Young's inequality to gain control on the remaining terms.
\begin{align}\label{Young}
    2f_j(f_{jk}(1-r_{jk}^{\infty})-s_{jk}^{\infty}g_{jk})+2g_j(g_{jk}(1-r_{jk}^{\infty})+f_{jk}s_{jk}^{\infty})\leq \frac{1}{a_{jk}}(f_{jk}^2+g_{jk}^2)+2a_{jk}((1-r_{jk}^{\infty})^2+(s_{jk}^{\infty})^2)(f_j^2+g_j^2),
\end{align}
where $a_{jk}$ are coefficients to be chosen later arising from the Young's inequality with $\epsilon$.
Then using the fact that $z_{jk}^{\infty}=e^{i(\a_j-a_k)}$ we have $(r_{jk}^{\infty})^2+(s_{jk}^{\infty})^2=1$ and we get
\begin{align}
    \ddt \mathcal{L}\leq\frac{\k
    }{N^2}\sum_{j,k=1}^N(f_{jk}^2+g_{jk}^2)(f_j-r_j^{\infty}+f_k-r_k^{\infty}+\frac{1}{a_{jk}})+4a_{jk}(1-r_{jk}^{\infty})(f_j^2+g_j^2).
\end{align}
Now we must choose $\k$ and $a_{jk}$ appropriately so that $(f_j-r_j^{\infty}+f_k-r_k^{\infty}+\frac{1}{a_{jk}})<0$ for all $t>T$ and so that
\begin{align}\label{i:min1}
    4a_{jk}(1-r_{jk}^{\infty})\leq C
\end{align}
where
\begin{align}\label{i:min2}
    f_j-r_j^{\infty}+f_k-r_k^{\infty}+\frac{1}{a_{jk}}+C<0,
\end{align}
in which case since $\frac{1}{N}\sum_{j=1}^Nf_j^2+g_j^2\leq \frac{1}{N^2}\sum_{j,k=1}^Nf_{jk}^2+g_{jk}^2$, by Cauchy-Schwartz, and we would achieve
\begin{align}
   \ddt \mathcal{L}\leq -c\k\mathcal{L}.
\end{align}

In other words the following inequality must hold
\begin{align}\label{0ineq}
    -1+a_{jk}(r_j^{\infty}+r_k^{\infty})-4a_{jk}^2(1-r_{jk}^{\infty})>0.
\end{align}
This inequality can be resolved within certain regimes. In particular with $\k>(2\sqrt{\frac{11+8\sqrt{2}}{7}})\max |\O_j|$, Corollary \ref{l:rbound2} guarantees $r_{jk}^{\infty}> 2(\sqrt{2}-1)$, and hence $\tilde r_j^{\infty}> 2(\sqrt{2}-1)$, as well. The inequality that must be satisfied becomes
\begin{align}\label{discr}
    -1+2a_{jk}r_{jk}^{\infty}-4a_{jk}^2(1-r_{jk}^{\infty})>0,
\end{align}
which has discriminant $\Delta= (r_{jk}^{\infty})^2-4(1-r_{jk}^{\infty})>0$ when $r_{jk}^{\infty}>2(\sqrt{2}-1)$.

Now, for $\k>(2\sqrt{\frac{11+8\sqrt{2}}{7}})\max_j|\O_j|$, we know that $r_{jk}^{\infty}\geq 2(\sqrt{2}-1)+\d$ for some $\d>0$ small, and for all $j,k=1,...,N$. Therefore, the discriminant is positive and there exists a nonempty interval $I=(R_1,R_2)$ where if $a_{jk} \in I$, then inequality \eqref{discr} is satisfied. In particular, if $\k=A\max_j|\O_j|$ for $A>2\sqrt{\frac{11+8\sqrt{2}}{7}}$, then the interval $I$ is given by solving the following inequality
\begin{align}
    -1+2a_{jk}\left(1-\frac{4}{A^2\left(1+\sqrt{1-\frac{4}{A^2}}\right)}\right)-4a_{jk}^2\left(\frac{4}{A^2\left(1+\sqrt{1-\frac{4}{A^2}}\right)}\right)>0.
\end{align}
Letting $B=\frac{1}{A^2\left(1+\sqrt{1-\frac{4}{A^2}}\right)}$ and solving the quadratic yields
\begin{align}
    B_{1,2}=\frac{1-4B\pm\sqrt{1-24B+16B^2}}{16B}.
\end{align}
Note that the condition $A>2\sqrt{\frac{11+8\sqrt{2}}{7}}$ is precisely the condition that guarantees $1-24B+16B^2>0$ so that $B_{1,2}$ exist and are distinct. Therefore in terms of $B$ the interval is given by
\begin{align}
    I_B=\left(\frac{1-4B-\sqrt{1-24B+16B^2}}{16B},\frac{1-4B+\sqrt{1-24B+16B^2}}{16B}\right).
\end{align}
 Further, note that as $A$ increases, the width of the interval increases as well, indeed as $A\to \infty$ (corresponding to $\k\to\infty$), we see that $B\to 0$ and the corresponding endpoints of the intervals converges to
\begin{align}
    I_{\infty}=\left(\lim_{B \to 0}\frac{1-4B-\sqrt{1-24B+16B^2}}{16B}, \lim_{B \to 0}\frac{1-4B+\sqrt{1-24B+16B^2}}{16B}\right)=(\frac{1}{2},\infty)
\end{align}
With this in hand we choose $a_{jk} \in I_B$ for the Young's inequality performed at \eqref{Young}. However, in order to make the optimal choice, we should minimize the following quantity:
\begin{align}\label{minimizer}
    \min_{a_{jk}\in I_B} 16Ba_{jk}+2(4B-1)+\frac{1}{a_{jk}}.
\end{align}
The expression in \eqref{minimizer} is obtained by combining the terms in \eqref{i:min1} and \eqref{i:min2}, excluding for the moment $f_j$ and $f_k$. This will give the negative contribution for the Gr\"onwall inequality, which is why we seek to minimize this quantity.

Differentiating with respect to $a_{jk}$ yields
\begin{align}
    \frac{d}{da_{jk}}\left(16Ba_{jk}+2(4B-1)+\frac{1}{a_{jk}}\right)=16B-\frac{1}{a_{jk}^2},
\end{align}
The minimum occurs at $a_{jk}=\frac{1}{4\sqrt{B}}$, whenever $\frac{1}{4\sqrt{B}}\in I_B$. However, with $B=\frac{1}{A^2\left(1+\sqrt{1-\frac{4}{A^2}}\right)}$ and $A>2\sqrt{\frac{11+8\sqrt{2}}{7}}$, we know that $B<\frac{1}{4}(3-2\sqrt{2})$ and therefore $\frac{1}{4\sqrt{B}} \in I_B$ for all $B<\frac{1}{4}(3-2\sqrt{2})$.

Therefore, $a_{jk}$ should be chosen as
\begin{align}
    a_{jk}=\frac{1}{4\sqrt{B}}
\end{align}
 in order to optimize the amount of space needed in the Gr\"onwall estimate, as well as the convergence rate. (We note in passing that all $a_{jk}$ have been chosen to be identical here, however, it is theoretically possible that choosing $a_{jk}$ as functions of $B_{jk}$ for each individual $j$ and $k$ could improve the estimates, allowing for slightly smaller values of $\k$, but believe this would complicate the argument more than necessary). The inequality then becomes,
\begin{align}\label{Gron1}
    \ddt \mathcal{L}&\leq\frac{\k}{N^2}\sum_{j,k=1}^N\left(f_{jk}^2+g_{jk}^2)(f_j+f_k+2(4B+4\sqrt{B}-1)\right),\\
    &=\frac{\k}{N^2}\sum_{j,k=1}^N\left(f_{jk}^2+g_{jk}^2)(f_j+f_k-C_B\right).
\end{align}

To conclude the proof of the lemma we note that at time $t=0$, if $\sup_j f_j(0)=M<\frac{C_B}{2}$, we have exactly
\begin{align}
    \ddt \cL|_{t=0}&\leq \frac{\k}{N^2}\sum_{j,k=1}^N\left(f_{jk}^2(0)+g_{jk}^2(0))(2M-C_B\right),\\
    &=-C_M\cL(0)
\end{align}
\end{proof}
At the moment we have proved that for $\sup_j f_j(0)=M<\frac{C_B}{2}$, then initially the Lyapunov candidate is decaying at an exponential rate. However, we have not established individual control over each $f_j(t)$, hence, it is possible for there to exist a time $T'$ such that $f_j(T')> \frac{C_B}{2}$, breaking the Gr\"onwall inequality. To further gain control over the individual $f_j$ we must further restrict the initial conditions with the following lemma.
\begin{lemma}
    Let $\k\geq A \max_j|\O_j|$ for $A>2\sqrt{\frac{11+8\sqrt{2}}{7}}$. Let $M_1=\sup_{j,k}|f_{jk}(0)|$ and $M_2=\sup_{j,k}|g_{jk}(0)|$. If 
    \begin{align}\label{Ndep}
        4N(M_1+M_2)-2M_1<C_B,
    \end{align} then for $C_{M_1}=2\k(C_B-2M_1)$, the following set
    \begin{align}
        I_{max}=\{t\geq0 : \ddt \cL \leq \frac{-C_{M_1}}{2}\cL\},
    \end{align}
    satisfies the following properties
    \begin{enumerate}
        \item $I_{max}$ is nonempty,
        \item $I_{max}$ is closed,
        \item $I_{max}$ is open,
    \end{enumerate}
    and hence $I_{max}=[0,\infty)$.
\end{lemma}
\begin{proof}
 Noting that $\sup_j f_j(0)=M\leq M_1<\frac{C_B}{2}$ along with Lemma \ref{l:decay} yields an initial rate of decay
 \begin{align}
     \ddt \cL|_{t=0}\leq -C_{M_1}\cL(0).
 \end{align}

Now observe that the set $I_{max}$ is nonempty due to the continuity of $\cL$. Second $I_{max}$ is closed is trivial. To see that it is also open, allowing us to conclude that $I_{max}=[0,\infty)$, we utilize the assumptions on $M_1$ and $M_2$.

Indeed, suppose that $4N(M_1+M_2)-2M_1<C_B$ and for the sake of contradiction that there is a $T^*$ finite such that $\ddt \cL(T^*)=-\frac{C_{M_1}}{2}\cL(T^*)$. Now as $\max_{jk}\frac{1}{2N^2} f_{jk}^2(t)\leq \cL(t)\leq \max_{jk} f_{jk}^2(t)+g_{jk}^2(t)$ always holds, and within $I_{max}$ we have $\ddt \cL \leq -\frac{C_{M_1}}{2}$, we have
\begin{align}
    \frac{1}{2N^2}\sup_{j}f_j^2(T^*)\leq \frac{1}{2N^2}\sup_{j,k} f_{jk}^2(T^*)\leq \cL(T^*)\leq e^{-\frac{C_{M_1}}{2}T^*}\cL(0)\leq \frac{1}{2}(M_1^2+M_2^2).
\end{align}
Hence, $\sup_j|f_j(T^*)|\leq N(M_1+M_2)$. Utilizing \eqref{Gron1} we also have
\begin{align}
    \ddt \cL(T^*)&\leq \frac{1}{2N^2}\sum_{j,k=1}^N (f_{jk}^2(T^*)+g_{jk}^2(T^*))2\k(2N(M_1+M_2)-C_B),\\
    &<\frac{1}{2N^2}\sum_{j,k=1}^N (f_{jk}^2(T^*)+g_{jk}^2(T^*))\k(2M_1-C_B),\\
    &=\frac{-C_{M_1}}{2}\cL(T^*).
\end{align}
Therefore $I_{max}=[0,\infty)$ and the system converges to the phase-locked state exponentially fast.
\end{proof}

Note the proof above gives the convergence rate as $-\frac{C_{M_1}}{2}$, however there was not anything special about the choice of $\frac{1}{2}$ in the definition of $I_{max}$ and the rate can be pushed as close to $C_{M_1}$ as we want. Further, as $C_{M_1}$ depends on the initial data, this provides an inverse relationship between the convergence rate and the admissible set of initial data around the fixed point for which the Lyapunov argument holds. Indeed, as $M_1$ increases, up to the maximal value $\frac{C_B}{2}$, we see that $C_{M_1}$ decreases. Further, larger values of $\k$ will also increase the synchronization rate, as $C_{M_1}=2\k(C_B-2M_1)$.

The relationship \eqref{Ndep} is also of particular interest. The inequality is required to ensure that the initial data is close enough to the fixed point so that the Gr\"onwall inequality is preserved in time. It is natural that a dependence on $N$ arises here as well as the need for $M_2$ to be small given that the Lyapunov candidate contains information on all $N$ wave functions and both the real and imaginary parts of the correlations.


\section{Asymptotic Wave Functions}\label{S:asym}
The result of Theorem \ref{Lyapunov} grants exponential convergence of the system of correlations \eqref{correlations} to the stable phase-locked state determined in Lemma \ref{l:phase-locked}. In this section we expand on what this means for the original wave functions $\psi_j(x,t)$ which solve the Schr\"odinger-Lohe system \eqref{SLmain}.

\begin{theorem}\label{t:asymptotic}
    Let $\psi_{j,0}(x)$, $j=1,...,N$ satisfy the assumptions of Theorem \ref{Lyapunov}. Then there exist wave functions $\tilde{\psi}_j$, $j=1,...,N$ such that
    \begin{align}
        \lim_{t\to\infty}\|\psi_j(t)-e^{-iHt}\tilde{\psi}_j\|_2=0,
    \end{align}
    where,
    \begin{align}
        \tilde{\psi}_j=\psi_{j,0}-i\int_0^{\infty} e^{-iHs}\left(\O_j\psi_j+\frac{i\k}{2}(\zeta-\langle \zeta,\psi_j\rangle\psi_j)\right)(s) \ds.
    \end{align}
\end{theorem}

Let us first begin with a useful lemma.
\begin{lemma}\label{L2-conv}
    Let $\psi_{j,0}(x)$, $j=1,...,N$ satisfy the assumptions of Theorem \ref{Lyapunov}. Then for each $j,k=1,...,N$,
    \begin{align}
        \lim_{t\to \infty}\|e^{i(\a_j-\a_k)}\psi_j(t)-\psi_k(t)\|_2=0.
    \end{align}
    Moreover the convergence rate is exponential and given by Theorem \ref{Lyapunov}, $E(t)\lesssim e^{-\frac{C_{M_1}}{2}t}$. As a consequence,
    \begin{align}\label{Edecay}
        \left\|\O_j\psi_j+\frac{i\k}{2}(\zeta-\langle \zeta,\psi_j\rangle \psi_j)\right\|_2\lesssim E(t).
    \end{align}
\end{lemma}
\begin{proof}
    From Theorem \ref{Lyapunov} we know that $|z_{jk} -e^{i(\a_j-\a_k)}|\lesssim E(t)$ for $E(t)\lesssim e^{-\frac{C_{M_1}}{2}t}$ an exponentially decaying quantity. Therefore since $\|\psi_j\|_2\equiv 1$,
    \begin{align}
        \|e^{i(\a_j-\a_k)}\psi_j(t)-\psi_k(t)\|_2^2&=2-2\Re\langle e^{i(\a_j-\a_k)}\psi_j(t)-\psi_k(t)\rangle,\\
        &=2-2\cos(\a_j-\a_k)r_{jk}-2\sin(\a_j-\a_k)s_{jk},\\
        &\lesssim E(t)
    \end{align}

    To see \eqref{Edecay} we utilize the fact that $\O_j=-\k s_j^{\infty}$ to combine the terms as
    \begin{align}
        \left\|\O_j\psi_j+\frac{i\k}{2}(\zeta-\langle \zeta,\psi_j\rangle \psi_j)\right\|_2&=\left\|\frac{i\k}{2}(\zeta-(\langle \zeta,\psi_j\rangle-2is_j^{\infty}) \psi_j)\right\|_2,\\
        &\lesssim \frac{\k}{2}\left\| \zeta-\bar{z}_j^{\infty}\psi_j\right\|_2+E(t)
    \end{align}
    Rewriting $\zeta$ and $\bar{z}_j^{\infty}$ in terms of individual indices $l$ we have
    \begin{align}
        \lesssim \frac{\k}{2N}\sum_{l=1}^N\left\| \psi_l-e^{i(\a_j-\a_l)}\psi_j\right\|_2+E(t)\lesssim E(t).
    \end{align}
\end{proof}
With this in hand, let us return to the integral formulation. Indeed any solution to \eqref{SLmain} satisfies
\begin{align}
    \psi_j(t)=e^{-iHt}\psi_{j,0}-i\int_0^t e^{-iH(t-s)}\left(\O_j\psi_j+\frac{i\k}{2}(\zeta-\langle \zeta,\psi_j\rangle \psi_j)\right) (s) \ds
\end{align}
Dividing through by $e^{-iHt}$ yields
\begin{align}
    e^{iHt}\psi_j(t)=\psi_{j,0}-i\int_0^t e^{iHs}\left(\O_j\psi_j+\frac{i\k}{2}(\zeta-\langle \zeta,\psi_j\rangle \psi_j)\right) (s) \ds
\end{align}
Due to the previous Lemma \ref{L2-conv}, we can conclude a strong $L^2$-limit for the above time integral,
\begin{align}
    \left\|\int_0^t e^{iHs}\left(\O_j\psi_j+\frac{i\k}{2}(\zeta-\langle \zeta,\psi_j\rangle \psi_j)\right) (s) \ds\right\|_2\lesssim \int_0^t E(s) \ds .
\end{align}
Therefore, we define
\begin{align}
        \tilde{\psi}_j=\psi_{j,0}-i\int_0^{\infty} e^{-iHs}\left(\O_j\psi_j+\frac{i\k}{2}(\zeta-\langle \zeta,\psi_j\rangle\psi_j)\right)(s) \ds.
    \end{align}
    for each $j=1,...,N$, and the result is now clear.\\

An important note is that this argument yields the asymptotic state of the original wave functions, $\psi_j$, as equal to the asymptotic state of new wave functions, $\tilde{\psi}_j$, under the action of the original Hamiltonian. However, this is to be expected due to the synchronization dynamics. Even though the convergence to phase-locked states occurs exponentially fast, the dynamics themselves distorts the structure of each wave function providing a fundamentally different form of the asymptotic wave functions $\tilde{\psi}_j$.

\section{Asymptotic Values}\label{AV}
In this section we examine the potential values that $\l$, $\k$ and $\a_i$ can take depending on the configuration of the natural frequencies. We drop the sub/superscripts of $\infty$ as we are only concerned with the asymptotic values and not the dynamics in the remainder of the paper. Let us begin with a quick lemma giving an upper bound on the order parameter
$\lambda$ in terms of the $\Omega_j$'s and $\kappa$.
\begin{lemma}
\begin{equation*}
\frac1N\sum_{j=1}^N(\sin\alpha_j)^2\leq 1-\lambda^2.
\end{equation*}
\end{lemma}
\begin{proof}
\begin{equation*}
\lambda=\frac1N\sum_{j=1}^N\cos\alpha_j\leq\frac1N N^{1/2}
\left(\sum_{j=1}^N\cos^2\alpha_j\right)^{1/2}
\end{equation*}
Consequently,
\begin{equation*}
\frac1N\sum_{j=1}^N\cos^2\alpha_j=
\frac1N\sum_{j=1}^N(1-\sin^2\alpha_j)\geq\lambda^2
\end{equation*}
\end{proof}
In the following subsections we provide an analysis of the asymptotic state that will be useful for providing stability analysis in Section \ref{S:Stable}. For this reason the results in this section are always in regards to what will be shown to be stable fixed points. Thus as for Kuramoto this will only occur when $\cos(\a_j)\geq 0$ for all $j=1,...,N$. Any fixed points with $\cos(\a_j)<0$ will automatically be saddle points as in \cite{MS}. Indeed, within \cite{MS}, the authors highlight two critical coupling values for the Kuramoto model $K_c$ and $K_c'$, where for $K<K_c$ no phase-locked states exist, for $K_c<K<K_c'$, phase-locked states exist, but are not stable, and for $K_c<K_c'<K$, there is a unique stable phase-locked state. Where $K_c=K_c'$ in the so-called degenerate case when $N=2$ or there is a perfectly symmetric distribution between an even number of oscillators. In what follows we focus on the critical coupling value $K_c'$, but for consistency in notation we will denote it by $\k^*$. We provide here a method for finding this critical value $\k^*$, and for providing bounds on the maximal angle of opening $\a_1-\a_N$.

\subsection{The two oscillator case}
For the Schr\"odinger-Lohe model, the case of $N=2$ oscillators has seen treatment in \cite{AM}, however, let us expand on the results here. Let $N=2$. In order for $\sum_i \O_i=0$, we have $\O_1=-\O_2=\O>0$. This implies that $\a_1=-\a_2=\a>0$, and $\l=\cos(\a)$. Now leaving $\O$ fixed, let us differentiate $\l$ with respect to the coupling strength $\k$.
\begin{align}
    \frac{d\l}{d\k}=-\sin(\a)\frac{d\a}{d\k}.
\end{align}
However, since we also have $\O=\l\k\sin(\a)$ so
\begin{align}
    \frac{d\l}{d\k}=\frac{-\O}{\k^2\sin^2(\a)}\left(\sin(\a)+\k\cos(\a)\frac{d\a}{d\k}\right),
\end{align}
as well. Equating these two and solving for $\frac{d\a}{d\k}$
gives
\begin{align}
    \frac{d\a}{d\k}=\frac{\O\sin(\a)}{\k(\k\sin^3(\a)-\O\cos(\a))}.
\end{align}
We can see here that the critical points of $\frac{d\a}{d\k}$ are when $\a=0$ (which only occurs if $\O=0$, or at $\k=\infty$), in which case $\frac{d\a}{d\k}=0$, and the other when
\begin{align}
    \k=\frac{\O\cos(\a)}{\sin^3(\a)}:=\k^*,
\end{align}
where $\frac{d\a}{d\k}$ fails to exist. At the critical value $\k^*$ we have
\begin{align}
    \O&=\l\k^*\sin(\a),\\
    &=\O\frac{\cos^2(\a)}{\sin^2(\a)},\\
    1&=\cot^2(\a).
\end{align}
implying $\a=\frac{\pi}{4}$, and thus $\l=\frac{\sqrt{2}}{2}$ and $\k^*=2\O$. We can now conclude that for $\k>2\O$, we have monotonicity of the fixed points, in that $\frac{d\a}{d\k}<0$ and $\frac{d\l}{d\a}>0$, and we see that as $\k \to \infty$ $\l \to 1$ and $\a \to 0$, recovering the practical synchronization result of \cite{CCH}. While for $\k<\k^*$ it is known that no asymptotic state exists \cite{AM}. This technique can be extended to the general case of finitely many oscillators, $N\geq 3$.\\

\subsection{The general finite $N$ oscillator case} For $N\geq 3$ we can still compute the derivative of the order parameter's $L^2$ norm with respect to the coupling strength. Indeed differentiating using both \eqref{lOmega} and \eqref{lcos},
\begin{align}
    \frac{d\l}{d\k}&=-\frac{1}{N}\sum_{j=1}^N\sin(\a_j)\frac{d\a_j}{d\k},\\
    &=\frac{-\O_j}{\k^2\sin^2(\a_j)}\left(\sin(\a_j)+\k\cos(\a_j)\frac{d\a_j}{d\k}\right),
\end{align}
for all $j=1,...,N$. As this is true for all $j$, we can for say $\a_1$, equate
\begin{align}
    \frac{-\O_1}{\k^2\sin^2(\a_1)}\left(\sin(\a_1)+\k\cos(\a_1)\frac{d\a_1}{d\k}\right)&=\frac{-\O_j}{\k^2\sin^2(\a_j)}\left(\sin(\a_j)+\k\cos(\a_j)\frac{d\a_j}{d\k}\right),
\end{align}
plugging in for $\O_1=\l\k\sin(\a_1)$ and $\O_j=\l\k\sin(\a_j)$,
\begin{align}
    \frac{-\l}{\k}-\l\cot(\a_1)\frac{d\a_1}{d\k}&=\frac{-\l}{\k}-\l\cot(\a_j)\frac{d\a_j}{d\k}.
\end{align}
Therefore,
\begin{align}
    \frac{d\a_j}{d\k}=\frac{\cot(\a_1)}{\cot(\a_j)}\frac{d\a_1}{d\k},
\end{align}
which can be plugged back into the equation for $\frac{d\l}{d\k}$ to get
\begin{align}
    \frac{d\l}{d\k}&=-\frac{1}{N}\sum_{j=1}^N\sin(\a_j)\frac{\cot(\a_1)}{\cot(\a_j)}\frac{d\a_1}{d\k},\\
    &=-\frac{\l}{\k}-\l\cot(\a_1)\frac{d\a_1}{d\k}.
\end{align}
Solving for $\frac{d\a_1}{d\k}$ yields
\begin{align}
    \frac{d\a_1}{d\k}=\frac{-\l\tan(\a_1)}{\k\left(\l-\frac{1}{N}\sum_{j=1}^N\frac{\sin^2(\a_j)}{\cos(\a_j)}\right)}.
\end{align}
and
\begin{align}
    \frac{d\l}{d\k}=-\frac{\l}{\k}+\frac{\l^2}{\k\left(\l-\frac{1}{N}\sum_{j=1}^N\frac{\sin^2(\a_j)}{\cos(\a_j)}\right)}
\end{align}
We now see the same type of behavior as for the $N=2$ case, where for $\k>\k^*$ the order parameter increases monotonically as $\k$ increases $(\frac{d\l}{d\k}>0)$, until at infinity $\l \to 1$ and all $\a_j \to 0$ monotonically in $\k$ as well. We can determine where the critical value of $\k^*$ is by looking at where $\frac{d\a_1}{d\k}$ fails to exist, at
\begin{align}
    \l&=\frac{1}{N}\sum_{j=1}^N\frac{\sin^2(\a_j)}{\cos(\a_j)},\\
    &=\frac{1}{N}\sum_{j=1}^N\frac{1-\cos^2(\a_j)}{\cos(\a_j)}=\frac{1}{N}\sum_{j=1}^N\cos(\a_j),
\end{align}
in other words, the critical value for synchronization is reached exactly when
\begin{align}\label{lstar}
    \l=\frac{1}{N}\sum_{j=1}^N\cos(\a_j)=\frac{1}{2N}\sum_{j=1}^N\frac{1}{\cos(\a_j)}:=\l^*.
\end{align}
With the formula \eqref{lstar} in hand we can pull out some interesting information about the synchronized state at the critical value $\k^*$. Indeed, the equation $\frac{1}{2\cos(x)}=\cos(x)$ in the region $-\frac{\pi}{2}<x<\frac{\pi}{2}$ has solutions $x=\pm\frac{\pi}{4}$. Which is to say, at the critical value $\k^*$ there exists an $i$ such that $|\a_i|<\frac{\pi}{4}$, if and only if there exists a $j$ such that $|\a_j|>\frac{\pi}{4}$. Meaning that at the critical coupling $\k^*$, the minimal angle of opening for the phase-locked state is $\frac{\pi}{2}$, achieved in the $N=2$ case. We express the upper and lower bounds on the order parameter $\l$ and the angle of opening $\a_1-\a_N$ that can be achieved at the critical value $\k^*$ in the following two lemmas:

\begin{lemma}\label{l:opening}
    At the critical value $\k=\k^*$ the angle of opening for the phase-locked state satisfies,
    \begin{align}
        \pi>2\cos^{-1}\left(\frac{-N+2+\sqrt{N^2-4N+36}}{8}\right)\geq\a_1-\a_N\geq \frac{\pi}{2},
    \end{align}
    where the value $\frac{\pi}{2}$ is only achieved when there are only two distinct natural frequencies so that for some $1<j<N$, we have $\O_1=...=\O_j=M_1$ and $\O_{j+1}=...=\O_N=-M_2$ with
    \begin{align}\label{eq:balance}
        jM_1=(N-j)M_2,
    \end{align}
    where \eqref{eq:balance} must hold due to the fact that $\sum_{j=1}^N \O_j=0$.
    Furthermore, the value $\l^*=\|\zeta_{\infty}\|_2$ at the critical value $\k^*$ satisfies,
    \begin{align}
       \frac{1+(N-1)^2}{N\sqrt{(N-1)^2+1}}\geq \l^*\geq \frac{\sqrt{2}}{2}
    \end{align}
    and is minimal only when $N$ is even and $j=\frac{N}{2}$, and maximal when $j=1$.
\end{lemma}
\begin{lemma}\label{k*bound}
    For $\k\geq\k^*$, the order parameter for the synchronized state has the lower bound $\l=\|\zeta_{\infty}\|_2\geq \frac{\sqrt{2}}{2}$ and angle of opening $\a_1-\a_N \leq 2\cos^{-1}\left(\frac{-N+2+\sqrt{N^2-4N+36}}{8}\right)$. Further, at the maximal angle of opening the value $\l^*$ can be computed directly as
    \begin{align}
        \l^*&=\frac{3(N-2)+\sqrt{N^2-4N+36}}{4N}
    \end{align}
\end{lemma}
\begin{proof}
    Let $\k=\k^*$, and $\O_1\geq \O_2\geq ... \geq  \O_N$, be such that $\sum_{j=1}^N\O_j=0$ and $\sum_{j=1}^N|\O_j|=M>0$. First the minimal case with only two distinct natural frequencies where $N=2n$ and $j=n$ is identical to the $N=2$ case already treated, yielding $\l^*=\frac{\sqrt{2}}{2}$ and $\a_1=\frac{\pi}{4}=-\a_N$. For a general distribution between the two distinct natural frequencies, the angle of opening will remain $\a_1-\a_N=\frac{\pi}{2}$, however the side with more oscillators will be shifted closer to the order parameter, increasing the value of $\l^*$. Indeed, if for some $1<j<N$, we have $\O_1=...=\O_j=M_1$ and $\O_{j+1}=...=\O_N=-M_2$ with
    \begin{align}
        jM_1=(N-j)M_2,
    \end{align}
    then using \eqref{lstar} we have
    \begin{align}\label{cos2}
        \left(\frac{j}{\cos(\a_1)}+\frac{N-j}{\cos(\a_N)}\right)=2(j\cos(\a_1)+(N-j)\cos(\a_N)),
    \end{align}
    where we used the fact that $\a_1=...=\a_j$ and $\a_{j+1}=...=\a_N$. Further we have
    \begin{align}
        j\sin(\a_1)+(N-j)\sin(\a_N)=0,
    \end{align}
    so that we can solve for $\a_N$ with respect to $\a_1$,
    \begin{align}
        \a_N&=\sin^{-1}\left(\frac{-j}{N-j}\sin(\a_1)\right),\\
        \cos(\a_N)&=\sqrt{1-\left(\frac{j}{N-j}\right)^2\sin^2(\a_1)}.
    \end{align}
    Plugging in for $\cos(\a_N)$ in \eqref{cos2} we can solve for $\cos(\a_1)$ and get $\cos(\a_1)=\frac{j}{\sqrt{(N-j)^2+j^2}}$ and therefore $\cos(\a_N)=\frac{N-j}{\sqrt{(N-j)^2+j^2}}$. In turn this yields $\cos^2(\a_1)+\cos^2(\a_N)=1$, and since $\a_1>0>\a_N$ we have $\a_1-\a_N=\frac{\pi}{2}$ as desired. Further, we see that $\l^*\geq\frac{\sqrt{2}}{2}$ as if $j\geq \frac{N}{2}$,
    \begin{align}
        \l^*&=\frac{1}{N}(j\cos(\a_1)+(N-j)\cos(\a_N)),\\
        &=\frac{j^2+(N-j)^2}{N\sqrt{(N-j)^2+j^2}},\\
        &\geq \frac{\sqrt{2}}{2}.
    \end{align}
    where equality only happens if $j=\frac{N}{2}$, further the largest $\l^*$ can be is exactly when $j=1$ so that,
    \begin{align}
        \l^*=\frac{1+(N-1)^2}{N\sqrt{(N-1)^2+1}}.
    \end{align}\\
    
    To get the maximal angle of opening, this occurs when $\O_1=\frac{M}{2}=-\O_N$ and $\O_j=0$ for $j\neq 1,N$. In this case $\a_j=0$ for all $j\neq 1,N$ and each of these $N-2$ oscillators are perfectly synchronized with the order parameter. Further symmetry gives $\a_1=-\a_N$ so that the equation for $\l^*$ becomes
    \begin{align}
        \l^*=\frac{1}{N}(2\cos(\a_1)+N-2)=\frac{1}{2N}\left(\frac{2}{\cos(\a_1)}+N-2\right).
    \end{align}
    Solving this for $\cos(\a_1)$ yields
    \begin{align}
        \cos(\a_1)=\frac{-N+2+\sqrt{N^2-4N+36}}{8}.
    \end{align}
    Returning to the angle of opening, as $\cos(\a_1)=\cos(\a_N)$, we have
    \begin{align}
    \a_1-\a_N&=2\cos^{-1}\left(\frac{-N+2+\sqrt{N^2-4N+36}}{8}\right).
    \end{align}
    Further, notice that the order parameter for this case is given by
    \begin{align}
        \l^*&=\frac{3(N-2)+\sqrt{N^2-4N+36}}{4N}<\frac{1+(N-1)^2}{N\sqrt{(N-1)^2+1}}, \ \ \text{for} \ N>\frac{7}{3}
    \end{align}
    the upper bound. Note $N=2$ returns exactly $\a_1-\a_2=\frac{\pi}{2}$ and $\l^*=\frac{\sqrt{2}}{2}$ as was determined previously. All other distributions of natural frequencies $\O_j$ at the critical threshold $\k^*$ produce synchronized states that lie within these bounds.
    
\end{proof}
As a sequence of corollaries we can further determine exactly the critical threshold for each of these configurations, which in turn gives a bound on $\k^*$ which depends on the number of wave functions $N$ and the mass of the natural frequencies $M=\frac{1}{N}\sum_{j=1}^N|\O_j|$.
\begin{corollary}
    For two distinct natural frequencies so that for some $1<j<N$, we have $\O_1=...=\O_j=M_1$ and $\O_{j+1}=...=\O_N=-M_2$ with
    \begin{align}
        jM_1=(N-j)M_2,
    \end{align}
    that way $\sum_{j=1}^N\O_j=0$ and $\sum_{j=1}^N|\O_j|=M$, the critical threshold for synchronization states existing is $\k^*=\frac{NM}{2j(N-j)}$.
\end{corollary}
As we have already computed at the threshold that $\cos(\a_1)=\frac{j}{\sqrt{(N-j)^2+j^2}}$ and $\cos(\a_N)=\frac{N-j}{\sqrt{(N-j)^2+j^2}}$, we further get,
\begin{align}
    \sin(\a_1)&=\cos(\a_N),\\
    \sin(\a_N)&=-\cos(\a_1).
\end{align}
Therefore,
\begin{align}
    \O_1&=\l^*\k^*\sin(\a_1),\\
    \k^*&=\frac{N\O_1}{N-j}=\frac{NM}{2j(N-j)}.
\end{align}

\begin{corollary}
    For $\O_1=\frac{M}{2}=-\O_N$ and $\O_j=0$ for all  $j\neq 1,N$, we can compute the critical threshold as
    \begin{align}
        \k^*=\frac{16NM}{(3(N-2)+\sqrt{N^2-4N+36})\sqrt{24+8N-2N^2+2(N-2)\sqrt{N^2-4N+36}}}
    \end{align}
\end{corollary}

\begin{corollary}
    For $N$ fixed and $\O_1\geq ...\geq \O_N$ such that $\sum_{j=1}^N\O_j=0$ and $\sum_{j=1}^N|\O_j|=M$, the critical value $\k^*$ can be found between the following bounds, depending on the distribution of the natural frequencies $\O_j$,
    \begin{align}
        \frac{2M}{N}\leq \k^*\leq \frac{16NM}{(3(N-2)+\sqrt{N^2-4N+36})\sqrt{24+8N-2N^2+2(N-2)\sqrt{N^2-4N+36}}}\leq M,
    \end{align}
    where all inequalities become equalities at $N=2$.
\end{corollary}

\section{Stability of the Phase Locked State}\label{S:Stable}
In the Section \ref{S:Lyapunov}, convergence to the phase-locked state was provided within certain regimes. In particular it was necessary that $\k>2\sqrt{\frac{11+8\sqrt{2}}{7}}\max_j |\O_j|$. This condition is not sharp due to the fact that we lost some space by utilizing Young's inequality. However, thanks to Theorem \ref{Kur-SL} we know that fixed points of \eqref{correlations} exist exactly when fixed  points of the Kuramoto model exist. In this section we also confirm the stability analysis of such fixed points. We proceed via analysis of the Jacobian of the fixed point map.

Returning to the system of correlations, we rewrite it here,
\begin{align}\label{corr}
    \ddt z_{jk}&=i(\O_j-\O_k)z_{jk} +\frac{\k}{2N}\sum_{l=1}^N(z_{jl}+z_{lk})(1-z_{jk}), \ \ j,k=1,...,N.
\end{align}
Recall $z_{jj}\equiv 1$ as the mass is conserved and $z_{jk}=\bar{z}_{kj}$, so we have a system of $\frac{1}{2}N(N-1)$ unknowns. However, for symmetry and structural reasons we will consider the stability of the system of $N(N-1)$ unknowns, where we have only excluded $z_{jj}=1$ for each $j=1,...,N$ (Note that this exclusion essentially eliminates an eigenvalue $\l=0$ of multiplicity $N$ that is inherent to the system due to the rotational invariance of the original Schr\"odinger-Lohe model \eqref{eq:SL_or}). Fixed points of \eqref{corr} satisfy
\begin{align}\label{Fmap}
    F_{jk}(\bz)=-\frac{i(\O_j-\O_k)}{\k}z_{jk}-\frac{1}{2N}\sum_{l=1}^N(z_{jl}+z_{lk})(1-z_{jk})=0,
\end{align}
for each $j\neq k$. In this way we can consider the diagonal forcing term with the natural frequencies as a perturbation of the homogeneous system where all $\O_j\equiv 0$. As the dynamics governed by \eqref{corr} are holomorphic we can compute the Jacobian matrix of the $F$-map, using each $z_{jk}$ as a variable to get the following $N(N-1) \times N(N-1)$ matrix
\[
D_{\bz}F(\bz) = \begin{blockarray}{ccccccccc}
z_{12} & \dots & z_{1N} & z_{21} & \dots & z_{2N} & \dots & z_{NN-1}\\
\begin{block}{(cccccccc)c}
  d_{1,1} & \dots & \dots & \dots & \dots & \dots & \dots & d_{1,N(N-1)} & z_{12} \\
  \vdots & \ddots &  &  &   & & & \vdots & \vdots \\
  \vdots &  & \ddots &  &  &  & & \vdots & z_{1N} \\
 \vdots &  &  & \ddots &  &  & & \vdots & z_{21} \\
 \vdots &  &  &  & \ddots & &  & \vdots & \vdots \\
 \vdots &  &  &  &  &  \ddots & & \vdots & z_{2N} \\
 \vdots &  &  &  &  &   & \ddots & \vdots & \vdots \\
 d_{N(N-1),1} & \dots & \dots   & \dots & \dots & \dots & \dots & d_{N(N-1),N(N-1)} & z_{NN-1}\\
\end{block}
\end{blockarray}
\]
where each $d_{n,m}$ is computed as
\begin{align}
    d_{n,m}=\partial_{z_{jk}} F_{il}(\bz),
\end{align}
where $j=1$ for $n\in\{1,...,N-1\}$, then $j=2$ for $n\in\{N,...,2N-2\}$, and so on, while $k=2,3,4,...,N$ (skipping $k=1$)for $n=1,2,3,...,N-1$ and then $k=1,3,4,...,N$ (skipping $k=2$) for $n=N,N+1,...,2N-2$ and so on. The indices $i,l$ are found identically except they are given corresponding to the value of the column index $m$.\\

As a simple illustrative example the matrix for $N=3$ takes the following form
\[
D_{\bz}F(\bz) = \begin{blockarray}{ccccccc}
z_{12} & z_{13} & z_{21} & z_{23} & z_{31} & z_{32}\\
\begin{block}{(cccccc)c}
  d_{1,1} & \dots & \dots & \dots & \dots & d_{1,6} & z_{12}\\
  \vdots & \ddots &  & & & \vdots & \vdots \\
  \vdots &  & \ddots &  & & \vdots & \vdots \\
 \vdots &  &  & \ddots &  & \vdots & \vdots \\
 \vdots &  &   &   & \ddots & \vdots & \vdots \\
 d_{6,1} & \dots &  \dots & \dots & \dots & d_{6,6} & z_{32}\\
\end{block}
\end{blockarray}
\]
where for example we can see that
\begin{align}
    d_{2,3}=\partial_{z_{13}}F_{21}(\bz), \ \ d_{5,2}=\partial_{z_{31}}F_{13}(\bz), \ \ d_{4,4}=\partial_{z_{23}}F_{23}(\bz).
\end{align}
Returning to the genearl matrix, many of the elements will be zero, as if $j\neq i$ and $k\neq l$, then  $d_{n,m}=\partial_{z_{jk}} F_{il}(\bz)=0$.
The elements on the main diagonal will never be zero and are given by
\begin{align}
    d_{n,n}=\partial_{z_{jk}}F_{jk}(\bz)=-\frac{i(\O_j-\O_k)}{\k}+\frac{1}{2N}\sum_{l=1}^N(z_{jl}+z_{lk})-\frac{1}{N}(1-z_{jk}).
\end{align}

Each row and column will have $2(N-2)$ nonzero terms off the main diagonal given by
\begin{align}\label{eq:indexmatch}
    \partial_{z_{ji}}F_{jk}(\bz)=\frac{-1}{2N}(1-z_{jk})=\partial_{z_{ik}}F_{jk}(\bz).
\end{align}
where these terms correspond to the matching up of exactly one index, which can be seen in \eqref{eq:indexmatch}.\\

Note that in what follows we are looking at the eigenvalues of the these Jacobians in order to determine the linear stability. In the definition of \eqref{Fmap} we consider the negative of the usual fixed point map, so that stability is achieved when all eigenvalues have positive real part.

\subsection{Stability of Fixed points in the Homogeneous regime}
\begin{lemma}\label{l:homostable}
    In the unperturbed case where all $\O_j\equiv 0$, all of the fixed points of \eqref{corr} and their corresponding linear stability are given by
\begin{itemize}
    \item Stable; Fully Synchronized: $z_{jk}\equiv 1$ for all $j,k=1,...,N$, 
    \item Saddle points; Bipolar states: There exists a set of indices $\mathcal{I}$ such that if $j,k\in \mathcal{I}$, then $z_{jk}=1$, and if $j,k \not\in \mathcal{I}$, then $z_{jk}=1$, while if exactly one of $j$ or $k$ is in $\mathcal{I}$, then $z_{jk}=-1$.
    \item Repulsive; Incoherence: The wave functions are equally distributed so that $\zeta=0$, and $z_{jk}=e^{\frac{2\pi i}{N}(j-k)}$.
    \item Repulsive; Trivial: $z_{jk}\equiv 0$ for all $j\neq k$.
\end{itemize} 
\end{lemma}
In \cite{HHK} the authors proved uniform $L^2$ stability of the Fully Synchronized state. Further, they showed that (excluding initial data where $\psi_j(0)=\psi_k(0)$ for $j\neq k$) the only possible Bipolar state that can be reached is given with only one wave function within the set $\mathcal{I}$ and showed instability of this state via a reduction to the $N=2$ case. The content of Lemma \ref{l:homostable} provides the linear stability of each of the steady states studied in \cite{HHK}, while also providing the stability of the two repulsive states given by Incoherence and the Trivial solution.

\begin{proof}
Considering the system of equations \eqref{corr}, the trivial solution $z_{jk}=0$ for $j\neq k$ is clearly a solution. To see the rest of the configurations, letting $\O_j\equiv 0$ and considering Lemma \ref{l:phase-locked}, we see that the above configurations are all of the fixed points of \eqref{corr}.  

In the Fully Synchronized case the entries of the Jacobian matrix we just computed are given by
\begin{align}
    &d_{n,n}=1,\\
    &d_{n,m}=0, \ \ \text{for} \ n\neq m,
\end{align}
that is, $D_{\bz}F(\bz)=I$ and all eigenvalues $\mu_j=1$ and the system is stable.\\

In the Bipolar synchronized states, let us consider the case of $\mathcal{I}=1$ so that only one wave function is balanced on the other side. In this case we see that
\begin{align}
    &d_{n,n}=\frac{N-2}{N}, \ \ \text{if} \ j,k\neq 1,\\
    &d_{n,n}=-\frac{2}{N}, \ \ \text{if} \ j=1, \ \text{or} \ k=1.
\end{align}
The off diagonal terms for when $j,k\neq 1$  are all zero. While if $j=1$ or $k=1$, there are exactly two off diagonal terms in each of the corresponding rows and columns of the Jacobian matrix, given by $d_{n,m}=\frac{1}{N}$.

Indeed, by rearranging the rows and columns this means the Jacobian matrix is given by a two block diagonal matrix. The first block is an $(N-1)(N-2) \times (N-1)(N-2)$ negative Laplacian matrix, and the second  block is a positive diagonal with $d_{n,n}=\frac{N-2}{N}$. Therefore there exists one eigenvalue $\mu_0=0$, and $(N-1)(N-2)-1$ negative eigenvalues $\mu_j=-\frac{2}{N}$, and the rest of the eigenvalues positive given by  $\mu_j=\frac{N-2}{N}$. In the same way all Bipolar states are saddle points.\\

The Incoherent state yields a Jacobian matrix with the terms
\begin{align}
    d_{n,n}=-\frac{1}{N}(1-e^{\frac{2\pi i}{N}(j-k)})
\end{align}
on the main diagonal, and the nonzero terms on the offdiagonal by
\begin{align}
    d_{n,m}=-\frac{1}{2N}(1-e^{\frac{2\pi i}{N}(j-k)})
\end{align}
By adding and subtracting the value $(N-2)\left(\frac{1}{N}\left(1-e^{\frac{2\pi i}{N}(j-k)}\right)\right)$ on the main diagonal appropriately we reach a Laplacian matrix minus a diagonal matrix that is more than twice as large as the diagonal of the Laplacian. The Perron-Frobenius theory (see \cite{HJ}) gives all eigenvalues of such a matrix to have negative real part, making the incoherent state a repulsive fixed point.\\

Last, the trivial state where all $z_{jk}=0$ is important to check as a priori in $L^2$ it is possible for all wave functions to end up orthogonal to each other. However, just as in the incoherent state, we achieve a repulsive fixed point. The main diagonal is given by $d_{n,n}=-\frac{1}{N}$, and the nonzero off diagonal terms by $d_{n,m}=-\frac{2}{N}$. Therefore adding and subtracting $\frac{1}{N}(N-2)$ on the main diagonal again yields a Laplacian matrix minus a diagonal matrix that is more than twice as large as the diagonal of the Laplacian.\end{proof}

\subsection{Stability of the fixed points of the inhomogeneous regime} We will use the computation in the previous subsection as inspiration for computing the linear stability in the case of heterogeneous natural frequencies where $\O_j \not\equiv 0$. Within this section we will utilize theory on matrix and eigenvalue perturbations, see \cite{SS} for the relevant background.

\begin{theorem}\label{t:unique}
   Let $\varepsilon=\frac{1}{\k}$, then for $\varepsilon>0$ small enough, so that at least $\k>\k^*$ determined in Section \ref{AV}, then the fixed point of \eqref{corr} determined by Lemma \ref{l:phase-locked}, is linearly stable. The fixed point is given by
    \begin{align}
        z_{jk}=e^{i(\a_j-\a_k)},
    \end{align}
    where $\a_j\in (-\frac{\pi}{2},\frac{\pi}{2})$ satisfy 
    \begin{align}
        \a_j&=\sin^{-1}\left(\frac{\O_j}{\l\k}\right),\\
        \cos(\a_j)&=\sqrt{1-\left(\frac{\O_j}{\l\k}\right)^2}\geq 0,\\
        \l&=\frac{1}{N}\sum_{j=1}^N\cos(\a_j),
    \end{align} for each $j=1,...,N$. Under the extra Assumption \ref{assumption}, we can further compute the perturbed eigenvalues of the Jacobian matrix and conclude linear stability for all $\k\geq 2\sqrt{2}\max_{j}|\O_j|$. 
\end{theorem}

We now return to analyzing the Jacobian matrix of the F-map \eqref{Fmap}, in the case where $\O_j\neq 0$, $\sum_{j=1}^N\O_j=0$, and $\k \geq \k^*$.  We see that the Jacobian matrix at the fixed point $z_{jk}=e^{i(\a_j-\a_k)}$ has main diagonal,
\begin{align}
    d_{n,n}&=-\frac{i(\O_j-\O_k)}{\k}+\frac{\l}{2}\left(e^{-i\a_k}+e^{i\a_j})\right)-\frac{1}{N}\left(1-e^{i(\a_j-\a_k)}\right),\\
    &=\frac{\l}{2}\left(e^{i\a_k}+e^{-i\a_j})\right)-\frac{1}{N}\left(1-e^{i(\a_j-\a_k)}\right),
\end{align}
where we utilized the characterization in Lemma \ref{l:phase-locked} to simplify the expression. Further, there are $2(N-2)$ nonzero off diagonal terms in each column given by
\begin{align}
    d_{n,m}=-\frac{1}{2N}\left(1-e^{i(\a_j-\a_k)}\right),
\end{align}
where the indices $j,k$ agree with the indices $j,k$ of the diagonal element in that column. Again, many terms will be zero as if $j\neq i$ and $k\neq l$, then  $d_{n,m}=\partial_{z_{jk}} F_{il}(\bz)=0$.

We first add and subtract $1$ on the diagonal to rewrite the Jacobian as
\[
D_{\bz}F(\bz)=I+E,
\]
where the elements of $E$ are given as
\begin{align}
    e_{n,n}=-1+\frac{\l}{2}\left(e^{i\a_k}+e^{-i\a_j})\right)-\frac{1}{N}\left(1-e^{i(\a_j-\a_k)}\right)
\end{align}
and $e_{n,m}=d_{n,m}$ are unchanged. Indeed, as $\k\to \infty$, or $\e\to 0$, we have seen in Section \ref{AV} that $\l\to 1$, and $\a_j\to 0$. In this way every element of $E$ is small and approaching $0$ as the coupling strength $\k$ increases. Therefore the Jacobian can be viewed as a perturbation of the identity matrix. Due to the continuity of eigenvalues with respect to the elements of the matrix, we know that for some $\k$ large ($\e$ small) all eigenvalues will retain a positive real part, guaranteeing stability of the system.\\

Of interesting note is that the choice of using the Identity matrix was not required. Indeed, we can extract more information about the eigenvalues by making a different choice. Adding and subtracting $e^{i(\a_j-\a_k)}$ along the diagonal lets us reformulate the problem into a different matrix perturbation.
\[
D_{\bz}F(\bz)=A+E^*=\begin{pmatrix}
    \ddots & & \\
    &  e^{i(\a_j-\a_k)} & \\
    & &  \ddots\\
\end{pmatrix}+E^*,
\]
where $E^*$ has diagonal terms
\begin{align}
    e^*_{n,n}=-e^{i(\a_j-\a_k)}+\frac{\l}{2}\left(e^{i\a_k}+e^{-i\a_j}\right)-\frac{1}{N}\left(1-e^{i(\a_j-\a_k)}\right),
\end{align}
and nonzero off diagonal terms,
\begin{align}
    e^*_{n,m}=d_{n,n},
\end{align}
and zeros in the same places as $D_{\bz}F(\bz)$.

Indeed, the case $\O_j\equiv 0$, returns $e^{i(\a_j-\a_k)}=1$ and thus $A=I$ and $E^*=0$, that is simply to say that when $\k$ is large, the matrix $E^*$ is small and the Jacobian can be viewed as a perturbation of the matrix $A=\mathrm{diag}\{e^{i(\a_j-\a_k)}\}$. As we know the eigenvalues of $A$ are given by its diagonal elements $\mu_{jk}=e^{i(\a_j-\a_k)}$, and that each $\a_j$, for $j=1,...,N$ are determined by the parameters $\O_j$ and $\k$, in configurations where the following condition holds we can compute the perturbation of the eigenvalues directly.
\begin{assumption}\label{assumption}
Suppose for any $j_1,j_2$ and $k_1,k_2$, that
\begin{align}
    \O_{j_1}-\O_{k_1}\neq \O_{j_2}-\O_{k_2}
\end{align}
\end{assumption}
This extra assumption precludes having $\O_j=\O_k$, however it guarantees that all of the asymptotic values $e^{i(\a_j-\a_k)}$ will be distinct. Therefore the eigenvalues of $A$ are simple and the perturbed eigenvalues are given by
\begin{align}\label{eigpert}
    \tilde{\mu}_{jk}&=\mu_{jk}(A)+\frac{v_{jk}^TE^*v_{jk}}{v_{jk}^Tv_{jk}}+O(\|E^*\|^2),\\
    &=\frac{\l}{2}\left(e^{i\a_k}+e^{-i\a_j}\right)-\frac{1}{N}\left(1-e^{i(\a_j-\a_k)}\right)+O(\|E^*\|^2).
\end{align}
where $v_{jk}$ is the eigenvector of $A$ corresponding to $\mu_{jk}$. As $A$ is diagonal, notice therefore that $\frac{v_{jk}^TE^*v_{jk}}{v_{jk}^Tv_{jk}}$ simply picks out the diagonal element  of $E^*$. Now the real part will tell us the stability of the fixed point.
\begin{align}
    \Re(\tilde{\mu}_{jk})=\frac{\l}{2}(\cos(\a_k)+\cos(\a_j))-\frac{1}{N}(1-\cos(\a_j-\a_k))+\Re(O(\|E^*\|^2)).
\end{align}

If $\k\geq 2\sqrt{2}\max_{j}|\O_j|$, then Corollary \ref{l:rbound3} guarantees $\cos(\a_j-\a_k)\geq \frac{\sqrt{2}}{2}$ for all $j,k$ and hence $\a_j\in [-\frac{\pi}{4},\frac{\pi}{4}]$, for all $j=1,...,N$, then automatically the first order eigenvalue perturbation yields stability of the fixed point. Indeed, if $\cos(\a_j-\a_k)\geq \frac{\sqrt{2}}{2}$, then $\cos(\a_j)\geq \frac{\sqrt{2}}{2}$ and hence $\l\geq \frac{\sqrt{2}}{2}$, as well. Therefore,
\begin{align}
    \frac{\l}{2}(\cos(\a_k)+\cos(\a_j))-\frac{1}{N}(1-\cos(\a_j-\a_k)) \geq \frac{\sqrt{2}}{4}\left(\frac{\sqrt{2}}{2}+\frac{\sqrt{2}}{2}\right)-\frac{1}{N}=\frac{1}{2}-\frac{1}{N}.
\end{align}
As the stability is already known for $N=2$, we can conclude the stability as well for the case $N\geq 3$. Further, the monotonicity of the asymptotic values with respect to $\k$ guarantee the stability will not be lost as $\k$ increases.\\

In fact, we can relax $\k$ further so long as at the very least, $\k$ is large enough so that $\a_j \in [-\frac{\pi}{4},\frac{\pi}{4}]$, which would give some value smaller than $2\sqrt{2}\max_j|\O_j|$, which could be computed for specific configurations. It appears as though we can relax $\k$ even further utilizing the computations in the previous Section \ref{AV}. However, it is clear that if $\a_j$ and $\a_k$ simultaneously approach $\pm\frac{\pi}{2}$, respectively, then the quantity
\begin{align}
    \frac{\l}{2}(\cos(\a_k)+\cos(\a_j))-\frac{1}{N}(1-\cos(\a_j-\a_k)) \to -\frac{2}{N}.
\end{align}
This appears to imply that there exist values $\k>0$ such that fixed points of the system \eqref{corr} exist, but none of them are stable. However, if $\k>\k^*$, the critical coupling value from Section \ref{AV}, then there is reason to believe the stability can be shown all the way to the value $\k^*$. Indeed, utilizing the configuration which maximizes the angle of opening $\a_1-\a_N$ at $\k=\k^*$ in Lemma \ref{k*bound}, we find exactly that
\begin{align}
    \frac{\l}{2}(\cos(\a_k)+\cos(\a_j))-\frac{1}{N}(1-\cos(\a_j-\a_k))=0,
\end{align}
Due to the monotonicity proved in Section \ref{AV}, for all $\k>\k^*$, this implies that the expression
\begin{align}
    \frac{\l}{2}(\cos(\a_k)+\cos(\a_j))-\frac{1}{N}(1-\cos(\a_j-\a_k))>0.
\end{align}
However, this configuration relied on the assumption that $\O_j=0$ for all $j\neq 1,N$, which contradicts the extra Assumption \ref{assumption} that provides simple eigenvalues for the matrix $A$. Thus due to the multiplicity of the eigenvalues we can not utilize \eqref{eigpert} for this configuration.

\end{document}